\documentclass[11pt, notitlepage]{article}   

\usepackage{amsmath,amsthm,amsfonts}   

\usepackage{amscd}
\usepackage{amsfonts}
\usepackage{amssymb}
\usepackage{color}      
\usepackage{epsfig}
\usepackage{graphicx}           

\usepackage{epsfig}
\include{psfig}

\theoremstyle{plain}
\newtheorem{Thm}{Theorem}[section]

\newtheorem{Crlr}[Thm]{Corollary}
\newtheorem{Prop}[Thm]{Proposition}

\newtheorem{Rem}[Thm]{Remark}
\newtheorem{Conj}[Thm]{Conjecture}

\theoremstyle{definition}
\newtheorem{Def}[Thm]{Definition}

\def\finf{\mathop{{\rm I}\kern -.27 em {\rm F}}\nolimits}

\begin{document}

\title{On Zero Forcing Number of Functigraphs}

\author{{\bf Cong X. Kang} and {\bf Eunjeong Yi}\\
\small Texas A\&M University at Galveston, Galveston, TX 77553, USA\\
{\small\em $^1$kangc@tamug.edu}; {\small\em $^2$yie@tamug.edu}}

\maketitle

\date{}

\begin{abstract}
\emph{Zero forcing number}, $Z(G)$, of a graph $G$ is the minimum cardinality of a set $S$ of black vertices (whereas vertices in $V(G)\!\setminus\!S$ are colored white) such that $V(G)$ is turned black after finitely many applications of ``the color-change rule": a white vertex is converted black if it is the only white neighbor of a black vertex. Zero forcing number was introduced and used to bound the minimum rank of graphs by the ``AIM Minimum Rank -- Special Graphs Work Group". Let $G_1$ and $G_2$ be disjoint copies of a graph $G$ and let $f: V(G_1) \rightarrow V(G_2)$ be a function. Then a \emph{functigraph} $C(G, f)=(V, E)$ has the vertex set $V=V(G_1) \cup V(G_2)$ and the edge set $E=E(G_1) \cup E(G_2) \cup \{uv \mid v=f(u)\}$. For a connected graph $G$ of order $n \ge 3$, it is readily seen that $1+\delta(G) \le Z(C(G, \sigma)) \le n$ for any permutation $\sigma$; we show that $1+ \delta(G) \le Z(C(G, f)) \le 2n-2$ for any function $f$, where $\delta(G)$ is the minimum degree of $G$. We give examples showing that there does not exist a function $g$ such that, for every pair $(G,f)$, $Z(G)<g(Z(C(G,f)))$ or $g(Z(G))>Z(C(G,f))$. We further investigate the zero forcing number of functigraphs on complete graphs, on cycles, and on paths. 
\end{abstract}


\vspace{.05in}



\section{Introduction}

Let $G = (V(G),E(G))$ be a finite, simple, connected, and undirected graph of order $|V(G)|=n \ge 2$. For a given graph $G$ and $S \subseteq V(G)$, we denote by $\langle S \rangle$ the subgraph induced by $S$. For a vertex $v \in V(G)$, the \emph{open neighborhood of $v$} is the set $N_G(v)=\{u \mid uv \in E(G)\}$ and the \emph{closed neighborhood of $v$} is the set $N_G[v]=N_G(v) \cup \{v\}$. The \emph{degree} $\deg_G(v)$ of a vertex $v \in V(G)$ is the the number of edges incident with the vertex $v$ in $G$. We denote by $\delta(G)$ the \emph{minimum degree} of a graph $G$. The \emph{distance} between two vertices $v, w \in V(G)$, denoted by $d_G(v,w)$, is the length of the shortest path between $v$ and $w$. For other terminologies in graph theory, refer to \cite{CZ}.\\

The notion of a zero forcing set, as well  as the associated zero forcing number, of a simple graph was introduced by the ``AIM Minimum Rank -- Special Graphs Work Group" in~\cite{AIM} to bound the minimum rank of associated matrices for numerous families of graphs. Let each vertex of a graph $G$ be given one of two colors, ``black" and ``white" by convention. Let $S$ denote the (initial) set of black vertices of $G$. The \emph{color-change rule} converts the color of a vertex from white to black if the white vertex $u_2$ is the only white neighbor of a black vertex $u_1$; we say that $u_1$ forces $u_2$, which we denote by $u_1 \rightarrow u_2$. And a sequence, $u_1 \rightarrow u_2 \rightarrow \cdots \rightarrow u_{i} \rightarrow u_{i+1} \rightarrow \cdots \rightarrow u_t$, obtained through iterative applications of the color-change rule is called a \emph{forcing chain}. Note that, at each step of the color change, there may be two or more vertices capable of forcing the same vertex. The set $S$ is said to be \emph{a zero forcing set} of $G$ if all vertices of $G$ will be turned black after finitely many applications of the color-change rule. The \emph{zero forcing number} of $G$, denoted by $Z(G)$, is the minimum of $|S|$ over all zero forcing sets $S \subseteq V(G)$. \\

Since its introduction by the aforementioned ``AIM group", zero forcing number has become a graph parameter studied for its own sake, as an interesting invariant of a graph. In~\cite{iteration}, the authors studied the number of steps it takes for a zero forcing set to turn the entire graph black; they named this new graph parameter the \emph{iteration index} of a graph: from the ``real world" modeling (or discrete dynamical system) perspective, if the initial black set is capable of passing a certain condition or trait to the entire population (i.e. ``zero forcing"), then the iteration index of a graph may represent the number of units of time (anything from days to millennia) necessary for the entire population to acquire the condition or trait. Independently, Hogben et al. studied the same parameter (iteration index) in~\cite{proptime}, which they called \emph{propagation time}. It's also noteworthy that physicists have independently studied the zero forcing parameter, referring to it as the \emph{graph infection number}, in conjunction with the control of quantum systems (see \cite{p1}, \cite{p2}, and \cite{p3}). More recently in~\cite{dimZ, dimZ2}, the authors initiated a comparative study between metric dimension and zero forcing number for graphs. In \cite{pzf}, the authors also introduced a probabilistic theory of zero forcing in graphs. For more articles and surveys pertaining to the zero forcing parameter, see \cite{pathcover, min-degree, Z+e, dimZ, ZGbar, ZFsurvey, ZFsurvey2, cutvertex}.\\

Chartrand and Harary \cite{permutation} introduced ``permutation graphs" (or ``generalized prisms"). Hedetniemi \cite{H} introduced the ``function graph", which consists of two graphs (not necessarily identical copies) with a function relation between them. Independently, D\"{o}rfler \cite{WD} introduced a ``mapping graph", which consists of two disjoint identical copies of a graph and additional edges between the two vertex sets specified by a function. The ``mapping graph" is rediscovered and studied in \cite{functi}, where it is called a ``functigraph". For articles on functigraphs, see \cite{functidom} and \cite{functimetric}. We recall the definition of the functigraph.

\begin{Def}
Let $G_1$ and $G_2$ be disjoint copies of a graph $G$, and let $f: V(G_1) \rightarrow V(G_2)$ be a function. A \emph{functigraph} $C(G, f)=(V, E)$ consists of the vertex set $V=V(G_1) \cup V(G_2)$ and the edge set $E(G)=E(G_1) \cup E(G_2) \cup \{uv \mid v=f(u)\}$.
\end{Def}

In this paper, we study the zero forcing number of functigraphs. For a graph $G$ of order $n\ge 3$, it is readily seen that $1+\delta(G) \le Z(C(G, \sigma)) \le n$ for a permutation $\sigma$; we show that $1+ \delta(G) \le Z(C(G,f)) \le 2n-2$ for a function $f$, and the bounds are sharp. We give examples showing that there does not exist a function $g$ such that, for every pair $(G,f)$, $Z(G)<g(Z(C(G,f)))$ or $g(Z(G))>Z(C(G,f))$. Further, we give zero forcing number of functigraphs on complete graphs $K_n$, and we give bounds for zero forcing number of functigraphs on cycles $C_n$ and on paths $P_n$. 


\section{Bounds on Zero Forcing Number of Functigraphs}

The \emph{path cover number} $P(G)$ of $G$ is the minimum number of vertex disjoint paths, occurring as induced subgraphs of $G$, that cover all the vertices of $G$. Next, we recall the definition that is stated in \cite{pp}. A graph $G$ is a graph of \emph{two parallel paths} if there exist two independent induced paths of $G$ that cover all the vertices of $G$ and such that $G$ can be drawn in the plane in such a way that the two paths are parallel and the edges (drawn as segments, not curves) between the two paths do not cross. A simple path is not considered to be such a graph. A graph that consists of two connected components, each of which is a path, is considered to be such a graph.

\begin{Thm} \cite{AIM, pathcover, cutvertex} \label{pathcover}
\begin{itemize}
\item[(a)] \cite{pathcover} For any graph $G$, $P(G) \le Z(G)$.
\item[(b)] \cite{AIM} For any tree $T$, $P(T) = Z(T)$.
\item[(c)] \cite{cutvertex} For any unicyclic graph $G$, $P(G)=Z(G)$.
\end{itemize}
\end{Thm} 

\begin{Thm}\cite{min-degree}\label{mindegree}
For any graph $G$, $Z(G) \ge \delta(G)$.
\end{Thm}

\begin{Thm} \cite{AIM} \label{idupper}
For any graphs $G$ and $H$, $Z(G \square H) \le \min \{Z(G)|V(H)|, Z(H)|V(G)|\}$, where $G \square H$ denotes the Cartesian product of $G$ and $H$. 
\end{Thm}

\begin{Thm} \label{ObsZ} Let $G$ be a connected graph of order $n \ge 2$. Then
\begin{itemize}
\item[(a)] \cite{dimZ, cutvertex} $Z(G)=1$ if and only if $G=P_n$,
\item[(b)] \cite{cutvertex} $Z(G)=2$ if and only if $G$ is a graph of two parallel paths,
\item[(c)] \cite{dimZ, cutvertex} $Z(G)=n-1$ if and only if $G=K_n$.
\end{itemize}
\end{Thm}

\begin{Thm} \cite{Z+e} 
Let $G$ be any graph. Then
\begin{itemize}
\item[(a)] For $v \in V(G)$, $Z(G)-1 \le Z(G-\{v\}) \le Z(G)+1$.
\item[(b)] For $e \in E(G)$, $Z(G)-1 \le Z(G-e) \le Z(G)+1$.
\end{itemize}
\end{Thm}

\begin{Thm}\cite{cutvertex} \label{cutV} 
Let $G$ be a graph with cut-vertex $v \in V(G)$. Let $W_1, W_2, \ldots, W_k$ be the vertex sets for the connected components of $\langle V(G)\setminus \{v\}\rangle$, and for $1\le i \le k$, let $G_i$ = $\langle W_i \cup \{v\}\rangle$. Then $Z(G) \ge  [\sum_{i=1}^{k} Z(G_i)]-k+1$.
\end{Thm}

Next, we obtain general bounds for zero forcing number of functigraphs. If $G$ is a graph of order $2$, one can easily check that $Z(C(G, f))=2$ for any function $f$. So, we only consider a graph $G$ of order $n \ge 3$ for the rest of the paper. Notice that $V(G_1)$ forms a zero forcing set for a permutation graph $C(G,\sigma)$; this, together with Theorem \ref{mindegree}, implies the following

\begin{Crlr} \cite{lncs}\label{n}
Let $G$ be a graph of order $n \ge 3$, and let $\sigma: V(G_1) \rightarrow V(G_2)$ be a permutation. Then $1+ \delta(G) \le Z(C(G, \sigma)) \le n$.
\end{Crlr}

\begin{Thm}\label{functibounds}
Let $G$ be a graph of order $n \ge 3$, and let $f: V(G_1) \rightarrow V(G_2)$ be a function. Then $1+ \delta(G) \le Z(C(G, f)) \le 2n-2$. Both bounds are sharp.
\end{Thm}

\begin{proof}
First, noting that $C(G, f) \not\cong K_{2n}$ for any function $f$, $Z(C(G, f)) \le 2n-2$ by (c) of Theorem~\ref{ObsZ}. Next, we show that $Z(C(G, f)) \ge 1+ \delta(G)$. If $Range(f)=V(G_2)$, the result follows by Corollary~\ref{n}. So, we consider $Range(f) \subsetneq V(G_1)$. Assume, to the contrary, that $Z(C(G,f)) \le \delta(G)$. Then $Z(C(G, f))=\delta(G)$ by Theorem \ref{mindegree}, since $\delta(C(G,f)) \ge \delta(G)$. Let $S$ be a zero forcing set for $C(G, f)$ with $|S|=\delta(G)=\delta$. Then $S$ must contain a vertex $v \in V(G_2) \setminus Range(f)$ satisfying $\deg_{C(G,f)}(v)=\deg_{G_2}(v)=\delta(G)$, along with all but one vertex in $N_{C(G, f)}(v)$. For each $1 \le i \le \delta$, let $F_i: w_{i,1} \rightarrow w_{i,2} \rightarrow \ldots \rightarrow w_{i, k_i}$ be a forcing chain consisting of $k_i$ vertices; notice that $S=\{w_{1, 1}, w_{2,1}, \ldots, w_{\delta, 1}\}$ is a zero forcing set for $C(G, f)$ and each vertex in $C(G,f)$ must appear in a forcing chain. We make the following 

\vspace{.1in}

\emph{Claim:} At most $(\delta-1)$ vertices in $G_1$ are turned black after applying the color-change rule on $S$ as long as possible.

\vspace{.1in}

\textit{Proof of Claim:} Assume, to the contrary, that $\delta$ vertices in $G_1$ are turned black. Let $B=\{w_{1, j_1}, w_{2, j_2}, \ldots, w_{\delta, j_{\delta}} \}$ be the first $\delta$ vertices in $G_1$ that are turned black. Since $C(G, f)$ is not a permutation graph, there exists a vertex $z \in V(G_2)$ such that $|f^{-1}(z)| \ge 2$. Notice that $z \not\in \cup_{i=1}^{\delta}\{w_{i, 1}, w_{i, 2}, \ldots, w_{i, j_{i}}\}$; otherwise, once $z$ is turned black, $z$ has at least two white neighbors in $G_1$, and thus at most $(\delta-1)$ vertices in $G_1$ are turned black. Since there exists a $z-w_{\ell,\alpha}$ path (a path connecting $z$ and $w_{\ell, \alpha}$) in $G_2$ such that $w_{\ell, \alpha}$ is adjacent to a vertex in $B$, for some $\ell$ ($1 \le \ell \le \delta$) and for $w_{\ell, \alpha} \in \cup_{i=1}^{\delta}\{w_{i, 1}, w_{i, 2}, \ldots, w_{i, j_{i}}\}$, even after $w_{\ell, \alpha}$ is turned black, $w_{\ell, \alpha}$  has at least two white neighbors, $f^{-1}(w_{\ell, \alpha}) \in B$ and a white neighbor in $G_2$ along a $z-w_{\ell, \alpha}$ path; thus,  it is impossible that $\delta$ vertices in $G_1$ are turned black. $\Box$

\vspace{.1in}

Since each vertex in $G_1$ has degree at least $\delta+1$, by the Claim, no vertex in $G_1$ can force. So, $S$ fails to be a zero forcing set for $C(G,f)$ with $|S|=\delta$, and thus $Z(C(G,f)) \ge \delta(G)+1$.\\

For the sharpness of the lower bound, take $G=K_n$ with $f=\sigma$ a permutation; then $\delta(K_n)=n-1$ and $Z(C(K_n, \sigma))=n$ (see Proposition~\ref{perK}). For the sharpness of the upper bound, take $G=K_n$ with $f=f_0$ a constant function; then $Z(C(K_n, f_0))=2n-2$ (see Proposition~\ref{constantK}).~\hfill 
\end{proof}


\section{Examples on $Z(G)$ versus $Z(C(G, f))$}

In this section, we give examples of functigraphs showing that there does not exist a function $g$ such that, for every pair $(G,f)$, $Z(G)<g(Z(C(G,f)))$ or $g(Z(G))>Z(C(G,f))$. In \cite{lncs}, examples of permutation graphs showing that $|Z(G)-Z(C(G,f))|$ can be arbitrarily large were given. Here, we first give an example of non-permutation functigraph showing that $Z(G)-Z(C(G,f))$ can be arbitrarily large.

\begin{Rem} 
There exists a functigraph such that $Z(G)-Z(C(G,f))$ can be arbitrarily large (see Figure \ref{Zdown2}); notice that $Z(G)=5k$ by (b) of Theorem \ref{pathcover}, and $Z(C(G, f)) \le 4k$ since the solid vertices in Figure \ref{Zdown2} form a zero forcing set for $C(G,f)$.
\end{Rem}

\begin{figure}[htpb]
\begin{center}
\scalebox{0.38}{\input{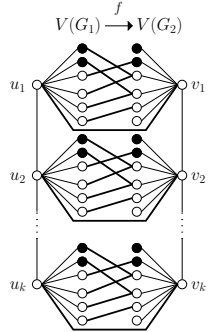}} \caption{An example showing that $Z(G)-Z(C(G,f))$ can be arbitrarily large}\label{Zdown2}
\end{center}
\end{figure}

Next, we give examples of functigraphs showing that there does not exist a function $g$ such that, for every pair $(G,f)$, $Z(G)<g(Z(C(G,f)))$ or $g(Z(G))>Z(C(G,f))$.

\begin{Rem} \label{refereeQ1}
There does not exist a function $g$ such that $Z(G)<g(Z(C(G,f)))$ for every pair $(G, f)$. Let $G$ be the graph in Figure \ref{ZPerdownBig} (a bouquet of $k$ circles), $V(G_1)=\{u_i \mid 0 \le i \le 2k\}$, $V(G_2)=\{v_i \mid 0 \le i \le 2k\}$, where $k \ge 3$. Let $f: V(G_1) \rightarrow V(G_2)$ be defined by $f(u_j)=v_j$ for $j \in \{0,1,2k\}$, $f(u_{2i})=v_{2i+1}$ for $1 \le i \le k-1$, and $f(u_{2i+1})=v_{2i}$ for $1 \le i \le k-1$. Note that $Z(G)=k+1$: $Z(G) \ge P(G)=k+1$ by (a) of Theorem \ref{pathcover}, and $Z(G) \le k+1$ since $\{u_0\} \cup \{u_{2i} \mid 1 \le i \le k\}$ is a zero forcing set for $G$. On the other hand, $Z(C(G,f)) \le 4$ since the solid vertices of Figure \ref{ZPerdownBig} form a zero forcing set for $C(G,f)$.
\end{Rem}

\begin{figure}[htpb]
\begin{center}
\scalebox{0.38}{\input{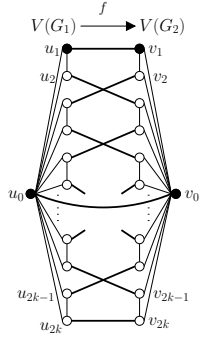}} \caption{An example showing that there does not exist a function $g$ such that $Z(G)<g(Z(C(G,f)))$ for every pair $(G,f)$}\label{ZPerdownBig}
\end{center}
\end{figure}

\begin{Rem} \label{refereeQ2}
There does not exist a function $g$ such that $g(Z(G))>Z(C(G,f))$ for every pair $(G,f)$. Let $G=P_{4k}$, $V(G_1)=\{u_i \mid 1 \le i \le 4k\}$, and $V(G_2)=\{v_i \mid 1 \le i \le 4k\}$, where $k \ge 2$. Let $f: V(G_1) \rightarrow V(G_2)$ be defined by $f(u_{2i-1})=v_{2i}$ and $f(u_{2i})=v_{2i-1}$, where $1 \le i \le 2k$ (see Figure \ref{refQ2}). Notice that $Z(G)=1$ by (a) of Theorem \ref{ObsZ}. On the other hand, $Z(C(G, f)) \ge k+1$ since at least a vertex in each $B_i=\{u_{4i-1}, u_{4i}, u_{4i+1}, u_{4i+2}, v_{4i-1}, v_{4i}, v_{4i+1}, v_{4i+2}\}$ ($1 \le i \le k-1$) and at least a vertex in each $\{u_j, u_{j+1}, v_j, v_{j+1}\}$ ($j=1, 4k-1$) must belong to a zero forcing set of $C(G,f)$; otherwise, a vertex in $B_i \setminus \{u_{4i-1}, u_{4i+2}, v_{4i-1}, v_{4i+2}\}$ or a vertex in $\{u_1, u_{4k}, v_1, v_{4k}\}$ fails to turn black, after applying the color-change rule as long as possible. 
\end{Rem}

\begin{figure}[htpb]
\begin{center}
\scalebox{0.38}{\input{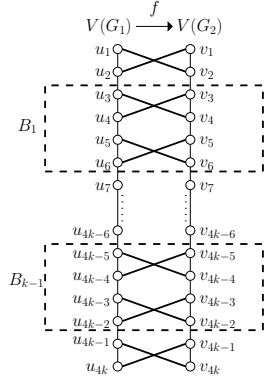}} \caption{An example showing that there does not exist a function $g$ such that $g(Z(G))>Z(C(G,f))$ for every pair $(G,f)$}\label{refQ2}
\end{center}
\end{figure}

In contrast to the above examples, as an immediate consequence of Theorem \ref{idupper}, we have the following

\begin{Crlr}\label{identityupper}
For any graph $G$ of order $n \ge 3$, $Z(C(G, id)) \le \min \{2Z(G), n\}$.
\end{Crlr}

Further, we make the folllowing

\begin{Conj}
For any graph $G$ of order $n \ge 3$, $Z(C(G, id)) \ge Z(G)+1$.
\end{Conj}


\section{Zero Forcing Number of Functigraphs on Complete Graphs}

In this section, for $n \ge 3$, we show that 1) $Z(C(K_n, \sigma))=n$ for a permutation $\sigma$; 2) $Z(C(K_n, f))=2n-1-|Range(f)|$ for a function $f$ with $1 \le |Range(f)|<n$. Throughout this section, we let $V(G_1)=\{u_i \mid 1 \le i \le n\}$ and $V(G_2)=\{v_i \mid 1 \le i \le n\}$ for $G_1 \cong G_2 \cong K_n$. 

\begin{Prop}\label{perK}
Let $G=K_n$ be the complete graph of order $n \ge 3$, and let $\sigma: V(G_1) \rightarrow V(G_2)$ be a permutation. Then $Z(C(K_n, \sigma))=n$.
\end{Prop}

\begin{proof}
Since $\delta(K_n)=n-1$, $Z(C(K_n, \sigma)) = n$ by Corollary \ref{n}. \hfill 
\end{proof}

\begin{Prop}\label{constantK}
Let $G=K_n$ be the complete graph of order $n \ge 3$, and let $f_0: V(G_1) \rightarrow V(G_2)$ be a constant function. Then $Z(C(K_n, f_0))=2n-2$.
\end{Prop}

\begin{proof}
Without loss of generality, we may assume that $f_0(u_i)=v_1$ for each $i$ ($1 \le i \le n$). If we let $H=\langle V(G_1) \cup \{v_1\}\rangle$, then $H \cong K_{n+1}$; one can view $C(K_n, f_0)$ as $H$ and $G_2$ being joined at the cut-vertex $v_1$. By (c) of Theorem \ref{ObsZ}, $Z(H)=n$ and $Z(G_2)=n-1$. By Theorem \ref{cutV}, $Z(C(K_n, f_0)) \ge 2n-2$. By Theorem \ref{functibounds}, $Z(C(K_n,f_0)) \le 2n-2$. Thus, $Z(C(K_n, f_0)) = 2n-2$.~\hfill
\end{proof}

\begin{Rem}
Theorem \ref{constantK} may be generalized as follows. For $m, n \ge 3$, let $H_1=K_m$ and $H_2=K_n$. Let $\mathcal{G}=(V,E)$ be the graph with $V=V(H_1) \cup V(H_2)$ and $E=E(H_1) \cup E(H_2) \cup \{u_iv_1 \mid 1 \le i \le m \}$. Then $Z(\mathcal{G})=m+n-2$.
\end{Rem}

\begin{Thm}
Let $G=K_n$ be the complete graph of order $n \ge 3$, and let $|Range(f)|=s$ where $1 < s <  n$. Then $Z(C(K_n, f)) = 2n-s-1$. 
\end{Thm}

\begin{proof}
Let $W=Range(f)$ with $|W|=s$, where $2 \le s \le n-1$. Without loss of generality, we may assume that $W=\{v_1, v_2, \ldots, v_s\}$ such that $|f^{-1}(v_i)|=k_i$ and $\sum_{i=1}^{s}k_i=n$. Further, we may assume that $f^{-1}(v_1)=\{u_i \mid 1 \le i \le k_1\}$, $f^{-1}(v_2)=\{u_i \mid k_1+1 \le i \le k_1+k_2\}, \ldots$, and $f^{-1}(v_s)=\{u_i \mid 1+ \sum_{t=1}^{s-1} k_{t} \le i \le \sum_{t=1}^{s} k_{t}\}$; we adopt the convention that $\sum_{i=a}^{b} f(i)=0$ when $b<a$. Let $S$ be a zero forcing set for $C(K_n,f)$. Note that $S=V(G_1) \cup \{v_i \mid s+1 \le i \le n-1\}$ is a zero forcing set for $C(K_n, f)$ with $|S|=n+(n-1-s)=2n-s-1$: (i) $u_{1+ \sum_{t=1}^{i-1}k_t} \rightarrow v_i$ for $1 \le i \le s$; (ii) $v_1 \rightarrow v_n$. So, $Z(C(K_n, f)) \le 2n-s-1$.\\

Next, we show that $Z(C(K_n, f)) \ge 2n-s-1$. Notice that no vertex in $V(G_2) \setminus W$ can force any vertex in $G_1$, and no vertex in $G_1$ can force any vertex in $V(G_2) \setminus W$. First, note that at least $(n-s-1)$ vertices of $V(G_2) \setminus W$ must belong to $S$; otherwise, even after all vertices of $W$ are turned black, each vertex in $G_2$ has at least two white neighbors in $G_2$, and hence it is impossible to turn the entire vertex set of $G_2$ black. Second, we make the following\\

\emph{Claim 1. For each $i$ ($1 \le i \le s$), $|S \cap f^{-1}(v_i)| \ge k_i-1$.}

\vspace{.1in} 

\emph{Proof of Claim 1:} Assume, to the contrary, that $|S \cap f^{-1}(v_i)| \le k_i-2$ for some $k_i \ge 2$ (there exists a $k_i \ge 2$ since $s <n$), where $1 \le i \le s$. Since $v_i$ has at least two white neighbors in $G_1$, $v_i$ cannot force at all. Further, even after all vertices in $V(G_1) \setminus f^{-1}(v_i)$ are turned black, noting that each vertex of $G_1$ has at least two white neighbors in $G_1$, it is impossible to turn the entire vertex set of $G_1$ black. This contradicts the assumption that $S$ is a zero forcing set. $\Box$

\vspace{.1in}
 
Third, we make the following

\vspace{.1in}

\emph{Claim 2. There exists at most one $j$ ($1 \le j \le s$) such that $|(\{v_j\} \cup f^{-1}(v_j)) \setminus S| >1$.} 

\vspace{.1in} 

\emph{Proof of Claim 2:} If $|(\{v_j\} \cup f^{-1}(v_j)) \setminus S| >1$ for some $j$ ($1 \le j \le s$), then, by Claim 1, $v_j \not\in S$ and $|f^{-1}(v_j) \setminus S|=1$. Suppose that there are two such $j$'s, say $j_1$ and $j_2$, where $1 \le j_1, j_2 \le s$. Then, even after all vertices in $V(G_1) \setminus (f^{-1}(v_{j_1}) \cup f^{-1}(v_{j_2}))$ are turned black, each vertex in $V(G_1) \setminus (f^{-1}(v_{j_1}) \cup f^{-1}(v_{j_2}))$ has at least two white neighbors in $G_1$; similarly, even after all vertices in $V(G_2) \setminus \{v_{j_1}, v_{j_2}\}$ are turned black, each vertex in $V(G_2) \setminus \{v_{j_1}, v_{j_2}\}$ has at least two white neighbors in $G_2$. So, it is impossible to turn the entire vertex set of $C(K_n,f)$ black, contradicting the assumption that $S$ is a zero forcing set. So, there exists at most one $j$ ($1 \le j \le s$) such that $|(\{v_j\} \cup f^{-1}(v_j)) \setminus S| >1$. $\Box$

\vspace{.1in}

If such a $j$ in Claim 2 exists, say $\{u, f(u)\} \cap S = \emptyset$ for some $u \in V(G_1)$, then either $V(G_1) \setminus \{u\} \subseteq S$ or $V(G_2) \setminus \{f(u)\} \subseteq S$; otherwise, each $G_i$ ($i=1,2$) has at least two white vertices, and thus no vertex in $S$ can force, contradicting the assumption that $S$ is a zero forcing set for $C(K_n, f)$. If $V(G_1) \setminus \{u\} \subseteq S$, then $S \cap W \neq \emptyset$ (otherwise, each vertex $u_x \in S \cap V(G_1)$ has two white neighbors, $u$ and $f(u_x)$, and thus $|S| \ge (n-1)+(n-s)=2n-s-1$. If $V(G_2) \setminus \{f(u)\} \subseteq S$, then, by Claim 1, $|S| \ge (n-1)+(n-s)=2n-s-1$. On the other hand, if $|(\{v_i\} \cup f^{-1}(v_i)) \setminus S| \le 1$ for each $i$ ($1 \le i \le s$), then $|S| \ge n+(n-s-1)=2n-s-1$. Thus, in each case, we have $|S| \ge 2n-s-1$. 

\vspace{.1in}

Therefore, $Z(C(K_n,f))=2n-s-1$ for $2 \le  s \le n-1$. \hfill
\end{proof}

\begin{figure}[htpb]
\begin{center}
\scalebox{0.4}{\input{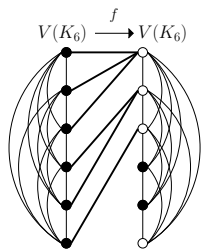}} \caption{$Z(C(K_6, f)) = 8$, where $|Range(f)|=3$}\label{ZnperK}
\end{center}
\end{figure}


\section{Zero Forcing Number of Functigraphs on Cycles}

In this section, for $n \ge 3$, we show that 1) $3 \le Z(C(C_n, \sigma)) \le n$ for a permutation $\sigma$; 2) $Z(C(C_n, f_0))=4$ for a constant function $f_0$; 3) $3 \le Z(C(C_n, f)) \le |Range(f)|+2$ for a function $f$ with $1<|Range(f)| <n$. Further, we give examples showing that the bounds of $Z(C(C_n, f))$ are sharp when $1<|Range(f)|<n$. Throughout this section, we let $V(G_1)=\{u_i \mid 1 \le i \le n\}$ and $E(G_1)=\{u_iu_{i+1} \mid 1 \le i \le n-1\} \cup \{u_1u_n\}$; similarly, let $V(G_2)=\{v_i \mid 1 \le i \le n\}$ and $E(G_2)=\{v_iv_{i+1} \mid 1 \le i \le n-1\} \cup \{v_1v_n\}$. 

\begin{Prop}\cite{AIM}\label{ZeroPC}
For $s \ge 3$ and $t \ge 2$,  $Z(C_s \square P_t) = \min\{s, 2t\}$.
\end{Prop}

As an immediate consequence of Proposition \ref{ZeroPC}, we have the following

\begin{Crlr}\label{zeroCid}
For $n \ge 3$, 
\begin{equation*} Z(C(C_n, id)) = \left\{
\begin{array}{ll}
3 & \mbox{ if $n=3$} ,\\
4 & \mbox{ if $n \ge 4$}  .
\end{array} \right.
\end{equation*}
\end{Crlr}

We recall that a graph $G$ is \emph{strongly regular} with parameters $(n,k,\alpha,\beta)$ if $|V(G)| = n$, $G$ is $k$-regular (i.e., the degree of each vertex in $G$ is $k$), every pair of adjacent vertices has $\alpha$ common neighbors, and every pair of non-adjacent vertices has $\beta$ common neighbors.

\begin{Prop} \cite{AIM}\label{srg}
If $G$ is a strongly regular graph, then $Z(G) \ge \lfloor \frac{|V(G)|}{2}\rfloor$.
\end{Prop}

\begin{Rem}\cite{AIM, lncs}\label{petersen}
The Petersen graph $\mathcal{P}$ (see Figure \ref{ZupPetersen}) is strongly regular; thus, $Z(\mathcal{P}) =5$ by Corollary \ref{n} and Proposition \ref{srg}. 
\end{Rem}

\begin{figure}[htpb]
\begin{center}
\scalebox{0.35}{\input{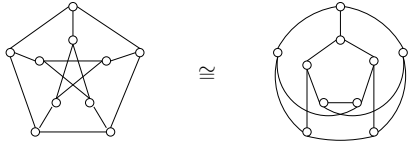}} \caption{The Petersen graph $\mathcal{P}$ with $Z(\mathcal{P})=5$}\label{ZupPetersen}
\end{center}
\end{figure}

As an immediate consequence of Corollary \ref{n}, we have the following

\begin{Crlr}
Let $G=C_n$ be the cycle of order $n \ge 3$, and let $\sigma: V(G_1) \rightarrow V(G_2)$ be a permutation. Then $3 \le Z(C(C_n, \sigma)) \le n$.
\end{Crlr}

\begin{Thm} \cite{lncs}
Let $G=C_n$ be the cycle of order $n \ge 3$, and let $\sigma: V(G_1) \rightarrow V(G_2)$ be a permutation. Then 
\begin{itemize}
\item[(a)] $Z(C(C_n, \sigma))=3$ if and only if $n=3$ (for any $\sigma$);
\item[(b)] $Z(C(C_n, \sigma))=n$ if and only if $n=3$ or $n=4$ (for any $\sigma$) or $C(C_n, \sigma)$ is isomorphic to the Petersen graph. 
\end{itemize}
\end{Thm}

\begin{Prop}
Let $G=C_n$ be the cycle of order $n \ge 3$, and let $f_0:V(G_1) \rightarrow V(G_2)$ be a  constant function. Then $Z(C(C_n,f_0)) =4$.
\end{Prop}

\begin{proof}
Without loss of generality, we may assume that $f_0(u_i)=v_1$ for each $i$ ($1 \le i \le n$). One can easily check that $S=\{u_1, u_2, v_2, v_3\}$ is a zero forcing set for $C(C_n, f_0)$, and thus $Z(C(C_n, f_0)) \le 4$. Next, we show that $Z(C(C_n, f_0)) \ge 4$. If we let $H=\langle V(G_1) \cup \{v_1\}\rangle$, then $Z(H)=3$: $Z(H) \le 3$ since $\{u_1, u_2, v_1\}$ is a zero forcing set for $H$, and $Z(H) \ge \delta(H)=3$ by Theorem \ref{mindegree}. Since $Z(G_2)=2$, noting that $v_1$ is a cut-vertex of $C(C_n, f_0)$, we have $Z(C(C_n, f_0)) \ge 4$ by Theorem \ref{cutV}. Thus, $Z(C(C_n,f_0))= 4$ for $n \ge 3$.\hfill
\end{proof}

\begin{Thm}\label{functiCbetween}
Let $G=C_n$ be the cycle of order $n \ge 3$. Let $f: V(G_1) \rightarrow V(G_2)$ be a function with $|Range(f)|=s$, where $1 < s < n$. Then $3 \le Z(C(C_n, f)) \le s+2$, and both bounds are sharp.
\end{Thm}

\begin{proof}
Let $W=Range(f)$ with $|W|=s$, where $2 \le s \le n-1$. By Theorem \ref{functibounds}, $Z(C(C_n, f)) \ge 3$. Next, we show that $Z(C(C_n,f)) \le s+2$. We consider two cases.

\vspace{.1in}

\emph{Case 1. $\langle W \rangle \cong s K_1$:} In this case, no two vertices in $W$ are adjacent in $G_2$; let $W=\{v_{j_1}, v_{j_2}, \ldots, v_{j_s}\}$, where $j_1 < j_2 < \ldots < j_s$. One can readily check that $S=\{u_1, u_2, v_k\} \cup (W \setminus \{v_{j_1}\})$, for $v_kv_{j_2} \in E(G_2)$ with $j_1 < k < j_2 $, forms a zero forcing set for $C(C_n, f)$ with $|S|=s+2$. 

\vspace{.1in}

\emph{Case 2. $\langle W \rangle \not\cong s K_1$:} In this case, there exist two adjacent vertices in $G_2$ that belong to $W$. One can readily check that $S=\{u_1, u_2\} \cup W$ forms a zero forcing set for $C(C_n, f)$ with $|S|=s+2$.

\vspace{.1in}

Thus, $Z(C(C_n,f)) \le s+2$. For the sharpness of the lower bound, let $f(u_i)=v_1$ for $1 \le i \le n-1$ and $f(u_n)=v_2$; then $Z(C(C_n, f))=3$ by Theorem \ref{functibounds} and the fact that $S=\{u_1,v_2, v_3\}$ is a zero forcing set for $C(C_n, f)$. For the sharpness of the upper bound, see Remark \ref{Csharp}.~\hfill
\end{proof}

\begin{Rem} \label{Csharp}
Let $G=C_{k^2}$, $V(G_1)=\{u_i \mid 1 \le i \le k^2\}$, and $V(G_2)=\{v_i \mid 1 \le i \le k^2\}$, where $k \ge 3$. Let $f:V(G_1) \rightarrow V(G_2)$ be defined by $f(u_{ak+i})=v_{i}$, where $1 \le i \le k$ and $0 \le a \le k-1$ (see Figure \ref{refQ}). Then $|Range(f)|=k$ and $Z(C(G, f)) = k+2$. 
\end{Rem}

\begin{figure}[htpb]
\begin{center}
\scalebox{0.38}{\input{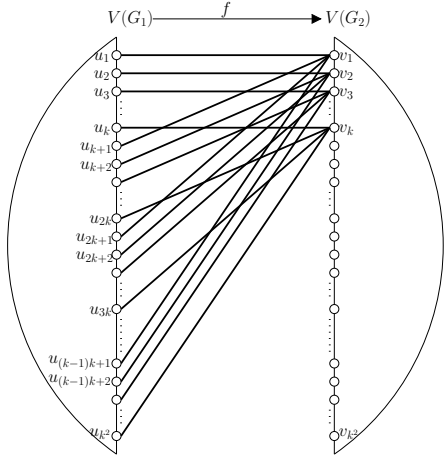}} \caption{An example satisfying $Z(C(C_n, f))=k+2$ with $|Range(f)|=k$}\label{refQ}
\end{center}
\end{figure}

\begin{proof}
Let $C(G, f)$ be the functigraph defined in Remark \ref{Csharp}. Since $S=\{u_1, u_2\} \cup \{v_i \mid 1 \le i \le k\}$ is a zero forcing set for $C(G, f)$ with $|S|=k+2$, $Z(C(G, f)) \le k+2$. Next, we show that $Z(C(G,f)) \ge k+2$. Assume, to the contrary, that $Z(C(G,f)) \le k+1$. Let $W=\{v_i \mid 1 \le i \le k\}$ and let $S$ be a zero forcing set for $C(G, f)$. First, notice that $|S \cap V(G_1)| \ge 2$; otherwise, no vertex in $G_2$ can force a vertex in $G_1$ and vice versa. Second, at least two adjacent vertices in $G_1$ must belong to $S$; otherwise, no vertex in $S \cap V(G_1)$ can force any vertex. Suppose that $|S \cap V(G_1)|=x$ and $|S \cap V(G_2)|=y$ with $x+y \le k+1$. Then at most ($x-2$) vertices in $S \cap V(G_1)$ forces their images, and thus at most $x+y-2$ ($\le k-1$) vertices in $W$ can be black after one global application of the color-change rule. So, there exists a vertex in $W$, say $v_j$ for some $j$ ($1 \le j \le k$), that is white unless a vertex in $f^{-1}(v_j)$ forces $v_j$ after applying the color-change rule on $S \cap V(G_1)$ as long as possible. For each $a$ ($0 \le a \le k-1$), $N_{G_1}[u_{ak+j}] \not\subseteq S$; otherwise, $u_{ak+j} \rightarrow v_j$, for some $a$, after one global application of the color-change rule. Even if $N_{G_1}[u_{ak+j}]$ (and thus $v_j$) are turned black, $v_j$ cannot force any vertex in $f^{-1}(v_j)$, unless all but one vertex in $f^{-1}(v_j)$ are black after applying the color-change rule as long as possible. So, $|S \cap \{u_{ak+i} \mid 1 \le i \le k\}| \ge 2$ for all but one value $a$ ($0 \le a \le k-1$), and at least a vertex in $u_{\alpha} \in S \cap V(G_1)$ satisfies $|S \cap N_{C(G,f)}[u_{\alpha}]| \ge 3$. Thus, $|S| \ge 2k-1$, contradicting the assumption that $|S| \le k+1$ for $k \ge 3$. Thus, $Z(C(G, f)) \ge k+2$, and thus $Z(C(G, f))=k+2$.~\hfill
\end{proof}


\section{Zero Forcing Number of Functigraphs on Paths}

In this section, for $n \ge 3$, we show that 1) $2 \le Z(C(P_n, \sigma)) \le n$ for a permutation $\sigma$; 2) $Z(C(P_n, f_0))=2$ for a constant function $f_0$; 3) $2 \le Z(C(P_n, f)) \le |Range(f)|+1$ for a function $f$ with $1<|Range(f)| <n$. Further, we give examples showing that the bounds of $Z(C(P_n, f))$ are sharp when $1<|Range(f)|<n$. Throughout this section, we let $V(G_1)=\{u_i \mid 1 \le i \le n\}$ and $E(G_1)=\{u_iu_{i+1} \mid 1 \le i \le n-1\}$; similarly, let $V(G_2)=\{v_i \mid 1 \le i \le n\}$ and $E(G_2)=\{v_iv_{i+1} \mid 1 \le i \le n-1\}$. 

\begin{Prop}\cite{AIM}\label{ZeroPP}
For $s,t \ge 2$,  $Z(P_s \square P_t) = \min\{s, t\}$.
\end{Prop}

As an immediate consequence of Proposition \ref{ZeroPP}, we have the following

\begin{Crlr}\label{zeroPid}
For $n \ge 3$, $Z(C(P_n, id)) =2$.
\end{Crlr}

As an immediate consequence of Corollary \ref{n}, we have the following

\begin{Crlr}
Let $G=P_n$ be the path of order $n \ge 3$, and let $\sigma: V(G_1) \rightarrow V(G_2)$ be a permutation. Then $2 \le Z(C(P_n, \sigma)) \le n$.
\end{Crlr}

\begin{Thm} \cite{lncs}
Let $G=P_n$ be the path of order $n \ge 3$, and let $\sigma: V(G_1) \rightarrow V(G_2)$ be a permutation. Then
\begin{itemize}
\item[(a)] $Z(C(P_n, \sigma))=2$ if and only if $C(P_n, \sigma) \cong P_n \square P_2$;
\item[(b)] $Z(C(P_n,{\sigma}))=n$ if and only if (i) $n=3$ and $C(P_3,{\sigma}) \not\cong P_3 \square P_2$, or (ii) $n=4$ and $C(P_4,{\sigma})$ is isomorphic to the permutation graph satisfying $\sigma(u_1)=v_3$, $\sigma(u_2)=v_4$, $\sigma(u_3)=v_1$, and $\sigma(u_4)=v_2$. 
\end{itemize}
\end{Thm}

\begin{Prop}
Let $G=P_n$ be the path of order $n \ge 3$, and let $f_0:V(G_1) \rightarrow V(G_2)$ be a  constant function. Then $Z(C(P_n,f_0)) =2$.
\end{Prop}

\begin{proof}
Since $C(P_n, f_0)$ is a graph of two parallel paths, $Z(C(P_n, f_0))=2$ by (b) of Theorem \ref{ObsZ}.~\hfill
\end{proof}

\begin{Thm}\label{functiPbetween}
Let $G=P_n$ be the path of order $n \ge 3$. Let $f: V(G_1) \rightarrow V(G_2)$ be a function with $|Range(f)|=s$, where $1 < s < n$. Then $2 \le Z(C(P_n, f)) \le s+1$, and both bounds are sharp.
\end{Thm}

\begin{proof}
Let $W=Range(f)$ with $|W|=s$, where $2 \le s \le n-1$. By Theorem \ref{functibounds}, $Z(C(P_n, f)) \ge 2$. Next, we show that $Z(C(P_n,f)) \le s+1$. We consider two cases.

\vspace{.1in}

\emph{Case 1. $\langle W \rangle \cong s K_1$:} In this case, no two vertices in $W$ are adjacent in $G_2$; let $W=\{v_{j_1}, v_{j_2}, \ldots, v_{j_s}\}$, where $j_1 <j_2 < \ldots < j_s$. One can readily check that $S=\{u_1, v_k\} \cup (W \setminus \{v_{j_1}\})$, for $v_kv_{j_2} \in E(G_2)$ with $j_1 <k<j_2$, forms a zero forcing set for $C(P_n, f)$ with $|S|=s+1$. 

\vspace{.1in}

\emph{Case 2. $\langle W \rangle \not\cong s K_1$:} In this case, there exist two adjacent vertices in $G_2$ that belong to $W$. One can readily check that $S=\{u_1\} \cup W$ forms a zero forcing set for $C(P_n, f)$ with $|S|=s+1$.

\vspace{.1in}

Thus, $Z(C(P_n,f)) \le s+1$. For the sharpness of the lower bound, let $f(u_i)=v_1$ for $1 \le i \le n-1$ and $f(u_n)=v_n$; then $Z(C(P_n, f))=2$ by (b) of Theorem \ref{ObsZ}. For the sharpness of the upper bound, consider the functigraph $C(P_{k^2}, f)$ with the function $f$ in Remark \ref{Csharp}, where $k \ge 3$; notice $|Range(f)|=k$, and one can check that $Z(C(P_{k^2}, f))=k+1$ using a similar argument as in the proof of Remark \ref{Csharp}.~\hfill
\end{proof}

\textit{Acknowledgement.} The authors greatly appreciate an anonymous referee for valuable comments, suggestions, and corrections on an earlier draft of this paper. In particular, they thank the said referee for suggesting the problems addressed in Remarks \ref{refereeQ1} and \ref{refereeQ2}, as well as for providing 
the example in Figure \ref{ZPerdownBig}.


\end{document}